\documentclass[12pt]{article}
\usepackage{epsfig}
\usepackage{a4wide}
\usepackage{latexsym}
\usepackage{amssymb}
\usepackage{subfigure}
\usepackage{amsmath}
\usepackage{amscd}
\usepackage{tabularx}

\usepackage{graphicx} 

\usepackage{subfigure}
\small\normalsize

\newcommand{\eop}{\qquad\hspace*{\fill}$\Box$}

\newtheorem{theorem}{Theorem}
\newtheorem{question}[theorem]{Question}

\newtheorem{corollary}[theorem]{Corollary}
\newtheorem{conjecture}[theorem]{Conjecture}

\newtheorem{observation}[theorem]{Observation}
\newtheorem{claim}[theorem]{Claim}

\def\cC{\mathcal{C}}
\def\cG{\mathcal{G}}
\def\cP{\mathcal{P}}
\def\cO{\mathcal{OP}}

\newenvironment{proof}{\begin{trivlist}\item[]{\bf Proof}\quad}%
  {\qquad\hspace*{\fill}\rule{1.2ex}{1.2ex}\end{trivlist}}
  {\qquad\hspace*{\fill}\rule{1.2ex}{1.2ex}\end{trivlist}}

\begin{document}

\title{{\bf Fire Containment in Planar Graphs}}

\author{Louis Esperet\hskip1pt\thanks{\,CNRS, Laboratoire G-SCOP, Grenoble,
    France; \texttt{louis.esperet@g-scop.fr}.}\hspace{15mm} Jan van den
  Heuvel\hskip1pt\thanks{\,Department of Mathematics, London School of
    Economics, London, UK; \texttt{jan@maths.lse.ac.uk}.}\\[2mm]
  Fr\'ed\'eric Maffray\hskip1pt\thanks{\,CNRS, Laboratoire G-SCOP,
    Grenoble, France;
    \texttt{frederic.maffray@g-scop.inpg.fr}.}\hspace{15mm} F\'elix
  Sipma\hskip1pt\thanks{\,Laboratoire G-SCOP, Grenoble, France;
    \texttt{felix.sipma@ens-lyon.org}.\vspace*{1mm}}}

\date{}
\maketitle

{\renewcommand{\thefootnote}{\relax}

  \footnotetext{This work was partially supported by ANR Project HEREDIA,
    under grant (\textsc{anr-10-jcjc-heredia}).\vspace*{1mm}}

  \footnotetext{Part of the research performed during a visit of JvdH to
    G-SCOP, Grenoble, supported by CNRS. JvdH likes to thank the members of
    G-SCOP for their hospitality.}}

\begin{abstract}
  \noindent
  In a graph $G$, a fire starts at some vertex. At every time step,
  firefighters can protect up to $k$ vertices, and then the fire spreads to
  all unprotected neighbours. The \emph{\mbox{$k$-surviving} rate}
  $\rho_k(G)$ of $G$ is the expectation of the proportion of vertices that
  can be saved from the fire, if the starting vertex of the fire is chosen
  uniformly at random. For a given class of graphs $\cG$ we are interested
  in the minimum value~$k$ such that for some constant $\epsilon>0$ and all
  $G\in\cG$, $\rho_k(G)\ge\epsilon$ (\emph{i.e.}, such that linearly many
  vertices are expected to be saved in every graph from~$\cG$).

  In this note, we prove that for planar graphs this minimum value is at
  most~4, and that it is precisely~2 for triangle-free planar graphs.

\end{abstract}

\section{Introduction}

The Firefighter Problem in graphs was introduced by Hartnell in
1995~\cite{Har95}. In a graph $G$, a fire starts at time 0 at some vertex
$v$ of $G$. At every subsequent time step, the firefighters protect at most
$k$ vertices from the fire (this protection is permanent), and then the
fire spreads to all unprotected neighbours. This problem has been heavily
studied over the past decade; we refer the reader to a survey of Finbow and
MacGillivray~\cite{FM09} for a general overview, and to
\cite{D11,FKMR07,KM10} for specific algorithmic and complexity results.

We denote by $sn_k(G,v)$ the maximum number of vertices of $G$ that can be
saved from the fire if it starts at $v$. In general, the parameter
$sn_k(G,v)$ depends heavily on $v$. For instance, if $G$ is a star on $n$
vertices, $v$ is the centre, and $u$ is a leaf, then $sn_1(G,v)=1$ while
$sn_1(G,u)=n-1$. A good indicator of the robustness of a graph against a
random fire is the following parameter introduced by Cai and Wang in
2009~\cite{CW09}. They define the \emph{$k$-surviving rate} $\rho_k(G)$ of
a graph $G$ on $n$ vertices as $\dfrac1{n^2}\sum\limits_{v\in G}sn_k(G,v)$.
In other words, $\rho_k(G)$ is the expectation of the proportion of
vertices that can be saved from the fire if it starts randomly in $G$.

\medskip
For a family $\cC$ of graphs, by a slight abuse of notation we use
$\rho_k(\cC)$ to denote the infimum of $\rho_k(G)$ over all graphs
$G\in\cC$ with at least two vertices. We also define the \emph{firefighter
  number} $\mathit{ff}(\cC)$ of the family $\cC$ as the minimum integer $k$
such that $\rho_k(\cC)>0$. If no such value $k$ exists, we set
$\mathit{ff}(\cC)=+\infty$.

Let $\cP$, $\cP_g$, and $\cO$ denote, respectively, the set of planar
graphs, the set of planar graphs with girth (size of a shortest cycle) at
least $g$, and the set of outerplanar graphs. Cai and Wang~\cite{CW09}
proved that $\rho_1(\cO)\ge1/6$, and asked the following.

\begin{question}{\rm\cite[Problem 6.2]{CW09}}\\*
  What is the minimum $k$ such that $\rho_k(\cP)>\epsilon$ for some
  $\epsilon>0$?
\end{question}

\noindent
Using the notation introduced above, this is equivalent to asking for the
value of $\mathit{ff}(\cP)$. Such a constant is at least two, as shown by
the complete bipartite (planar) graph $K_{2,n-2}$: When there is only one
firefighter, only two vertices can be saved wherever the fire starts, hence
$\rho_1(K_{2,n-2})=2/n$.

Wang \emph{et al.}~\cite{WFW10} recently proved that
$\rho_1(\cP_9)\ge2/35$, and that if a graph $G$ is $d$-degenerate (and has
at least two vertices), $\rho_{2d-1}(G)\ge2/(5\,d)$. This implies that the
firefighter number of every proper minor-closed class of graphs is finite.
Wang \emph{et al.}\ also proved that $\rho_5(\cP)\ge2/15$, which implies
that $2\le\mathit{ff}(\cP)\le5$.

The main purpose of this note is to prove the following results.

\begin{theorem}\label{th:total}\mbox{}\\*
  (1)\quad For the class $\cP$ of planar graphs, we have
  $2\le\mathit{ff}(\cP)\le4$.

  \noindent
  (2)\quad For the class $\cP_4$ of triangle-free planar graphs, we have
  $\mathit{ff}(\cP_4)=2$.
\end{theorem}

\noindent
The proofs of these results can be found in Sections~\ref{sec:pla}
and~\ref{sec:bip}. For the planar case, we indeed prove a much stronger
theorem: we show that if 4 firefighters are available at the first step,
and 3 firefighters at each subsequent step, then the surviving rate of
every planar graph is bounded by a positive constant.

The main idea of the two proofs is to partition the vertices of a graph $G$
into two carefully chosen sets $X$ and~$Y$. If the fire starts at a vertex
of $X$, we will show that it can be quickly contained (saving all vertices
but a constant number). If the fire starts at a vertex of $Y$, we will do
nothing and let everything burn. We could easily save a couple of vertices
by protecting them, but this would only make the computation harder (and
only improve the constants). Then we will show that $|Y|\le c\,|X|$ for
some constant $c>0$, concluding the proof.

We illustrate this technique by proving the easy result below. The proofs
in Sections~\ref{sec:pla} and~\ref{sec:bip} are more involved and use the
well-known discharging method for planar graphs. The novel aspect of our
approach is that we use the discharging method not just to prove that a
particular configuration must exist at least once, but that it must exist
\emph{many} times.

\begin{theorem}\label{th:girth5}\mbox{}\\*
  Any planar graph $G$ with girth at least~5 and at least two vertices
  satisfies $\rho_2(G)\ge1/22$.
\end{theorem}

\begin{proof}
  From Euler's formula, it is easy to deduce that planar graphs with girth
  at least~5 have average degree less than 10/3; while planar graphs with
  girth at least~6 have average degree less than 3.

  Let $X_2$ and $Y_4$ be the set of vertices of $G$ of degree at most 2 and
  at least 4, respectively. Let $X_3$ be the set of vertices of degree 3
  with a neighbour of degree at most 3, and let $Y_3$ be the set of
  vertices of degree 3 not in $X_3$. We use $x_2,x_3,y_3,y_4$ to denote the
  cardinality of the sets $X_2,X_3,Y_3,Y_4$, respectively. Let~$n$ be the
  number of vertices in~$G$.

  If the fire starts at $v\in X_2$, we protect its two neighbours, saving
  $n-1$ vertices. Consider a vertex $v\in X_3$, and let $u$ be its
  neighbour of degree at most~3. If the fire starts at $v$, we first
  protect its two neighbours distinct from $u$. The fire then reaches $u$,
  and we protect the two neighbours of $u$ distinct from $v$, saving $n-2$
  vertices. If the fire starts at a vertex from $Y_3$ or $Y_4$, we do
  nothing.

  Since for a fire that starts at a vertex from $X_2\cup X_3$ we can save
  at least $n-2$ vertices, we obtain for the 2-surviving rate
  \begin{equation}\label{eq:1}
    \rho_2(G)\:=\:\frac1{n^2}\sum_{v\in G}sn_2(G,v)\:\ge\:
    \frac1{n^2}\cdot(x_2+x_3)\,(n-2)\:=\:
    \frac{n-2}{n}\cdot\frac{x_2+x_3}{x_2+x_3+y_3+y_4}.
  \end{equation}

  Consider the subgraph $H$ of $G$ induced by the edges with one end in
  $Y_3$ and the other in $Y_4$. This graph has at most $y_3+y_4$ vertices
  and precisely $3\,y_3$ edges. Since $H$ is bipartite and $G$ has girth at
  least~5, $H$ has girth at least~6. Hence, its average degree is less than
  3 and we have $6\,y_3\le3\,(y_3+y_4)$, implying that $y_3\le y_4$.

  Since the average degree in~$G$ is less than 10/3,
  $3\,x_3+3\,y_3+4\,y_4\le\frac{10}3\,(x_2+x_3+y_3+y_4)$. Using that
  $y_3\le y_4$, this implies $y_4\le10\,x_2+x_3$, and hence
  $y_3+y_4\le20\,x_2+2\,x_3\le20\,(x_2+x_3)$. As a consequence, we obtain,
  using~\eqref{eq:1}:
  $$\rho_2(G)\:\ge\:\frac{n-2}{n}\cdot\frac{x_2+x_3}{x_2+x_3+y_3+y_4}\:
  \ge\:\frac{n-2}{n}\cdot\frac{1}{1+\frac{y_3+y_4}{x_2+x_3}}\:\ge\:
  \frac{n-2}{21\,n}\hskip1pt.$$
  If $n\ge44$, we obtain $\rho_2(G)\ge1/22$. Otherwise, if $G$ has only two
  vertices, the vertex distinct from the firestart can be saved, while if
  $G$ has $3\le n\le44$ vertices, at least $2/44=1/22$ of the vertices can
  be saved. So in all cases we have $\rho_2(G)\ge1/22$.
\end{proof}

\noindent
Theorem~\ref{th:girth5} has the following immediate consequence.

\begin{corollary}\mbox{}\\*
  For the class $\cP_5$ of planar graphs with girth at least~5, we have
  $1\le\mathit{ff}(\cP_5)\le2$.
\end{corollary}

\section{Planar graphs}\label{sec:pla}

In this section we prove the following theorem.

\begin{theorem}\label{th:pla}\mbox{}\\*
  Assume 4 firefighters are given at the first step, and then 3 at each
  subsequent step. Then the firefighters have a strategy such that every
  planar graph has surviving rate at least 1/2712.
\end{theorem}

\begin{proof}
  We can assume that $G$ is a maximal planar graph (hence a planar
  triangulation), since adding edges to the graph can only make things more
  difficult for the firefighters. Hence, $G$ has minimum degree at least 3.
  For $3\le d\le6$, let $X_d$ be the set of vertices $v$ of degree~$d$ so
  that if the fire starts at $v$, the firefighters have a strategy that
  saves at least $|V(G)|-6$ vertices; the other vertices of degree~$d$ form
  the set~$Y_d$. For $d\ge 7$, $Y_d$ is the set of all vertices of
  degree~$d$. We set $X=\bigcup_{3\le d\le6}X_d$ and
  $Y=\bigcup_{d\ge3}Y_d$.

  Note that every vertex $v$ of degree $3\le d\le4$ is in $X_d$, since
  placing the firefighters on~$v$'s neighbours saves all the vertices
  except $v$. Hence, $Y_3$ and $Y_4$ are both empty. Also observe that if a
  vertex $v$ of degree~5 has a neighbour $u$ of degree at most 6, then $v$
  is in~$X_5$: first place four firefighters on the neighbours of $v$
  distinct from $u$, the fire then spreads to~$u$. Since $G$ is maximal
  planar only three unprotected neighbours of $u$ remain, which can be
  protected by three firefighters in the next step.

  We now make a small observation about the sets $X_6$ and $Y_6$. The
  \emph{length} of a path is the number of edges on the path.

  \begin{observation}\label{obs:hex}
    For every vertex $v\in Y_6$, there is a path of length at most 3
    connecting $v$ and a vertex $u$ of degree distinct from 6, and such
    that all the internal vertices on the path have degree precisely 6.
  \end{observation}

  \noindent
  Assume this is not the case. Then the subgraph of $G$ induced by the
  vertices at distance at most 3 from $v$ is the induced subgraph of a
  hexagonal grid. In this case, Figure~\ref{fig:gridhex} depicts a strategy
  for the firefighters saving all the vertices except at most six (which
  contradicts the fact that $v\in Y_6$). The figure should be read as
  follows: the fire starts at the squared vertex labelled 0, then the
  firefighters protect the circled vertices labelled 1, the fire spreads to
  all the squared vertices labelled 1, the firefighters protect the circled
  vertices labelled 2, and so on.\eop

  \smallskip
  \begin{figure}[ht]
    \centering
      \includegraphics[scale=0.8]{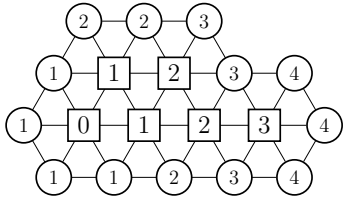}
      \caption{A strategy saving at least $n-6$ vertices if the
        neighbourhood of the firestart is a hexagonal
        grid.}\label{fig:gridhex}
  \end{figure}

  \bigskip\noindent
  For a planar graph with vertex set~$V$, edge set~$E$ and face set~$F$,
  Euler's formula gives $|V|-|E|+|F|=2$. For simple maximal planar graphs,
  it is well known that this is equivalent to
  $\sum\limits_{v\in V}(d(v)-6)=-12$. We interpret this by giving each
  vertex~$v$ an initial charge $\sigma_1(v)=d(v)-6$. We redistribute this
  initial charge according to the following rules. Here the value
  of~$\alpha$ will be determined later.

  \medskip\noindent
  (R1)\quad Each vertex of degree at least~7 gives a charge 1/4 to each of
  its neighbours from~$Y_5$.

  \smallskip\noindent
  (R2)\quad For each vertex $v\in Y_6$ we choose one vertex~$u$ with
  $d(u)\ne6$ and $\mathop{dist}(v,u)\le3$ (using
  Observation~\ref{obs:hex}); this vertex~$u$ gives a charge~$\alpha$
  to~$v$.

  \medskip
  The charge obtained after applying the rules~(R1) and~(R2) is denoted by
  $\sigma_2(v)$, $v\in V$. Note that we have
  $\sum\limits_{v\in V}\sigma_2(v)=\sum\limits_{v\in V}\sigma_1(v)=-12$.

  \begin{observation}\label{obs:R1}
    A vertex~$v$ with $d(v)\ge7$ has at most
    $\bigl\lfloor\tfrac12\,d(v)\bigr\rfloor$ neighbours in~$Y_5$.
  \end{observation}

  \noindent
  This follows directly, since if~$v$ has a neighbour~$u$ in~$Y_5$, then
  the two common neighbours of~$u$ and~$v$ must have degree at least~7 as
  well (otherwise~$u$ is in~$X_5$).\eop

  \begin{claim}\label{cl:ch}
    There is a constant $\alpha>0$ such that for every $v\in X$ we have
    $\sigma_2(v)>-3-93\,\alpha$; while for every $v\in Y$ we have
    $\sigma_2(v)\ge\alpha$.
  \end{claim}

  \noindent
  To prove the claim, we first estimate how often a vertex~$v$ with
  $d(v)\ne6$ can give a charge~$\alpha$ according to~(R2). As a very crude
  upper bound, this is at most the number of paths of length at most~3,
  starting in~$v$, and whose internal vertices all have degree exactly~6.
  This number is clearly at most $d(v)\cdot(1+5+25)=31\,d(v)$.

  Each vertex~$v$ of degree~3 gives at most $3\times31$ times a charge
  according to~(R2). Since $\sigma_1(v)=-3$, this gives
  $\sigma_2(v)\ge-3-93\,\alpha$. Similarly, for a vertex~$v$ of degree~4 we
  have $\sigma_2(v)\ge-2-124\,\alpha$; while for $v\in X_5$ we have
  $\sigma_2(v)\ge-1-155\,\alpha$. Finally, for a vertex $v\in X_6$ we have
  $\sigma_2(v)=\sigma_1(v)=0$.

  From rule~(R2) it follows that if $v\in Y_6$, then
  $\sigma_2(v)=\sigma_1(v)+\alpha=\alpha$.

  For a vertex~$v$ with degree $d(v)\ge7$, we can estimate, using the
  observations above,
  \[\sigma_2(v)\:\ge\:(d(v)-6)-
  \bigl\lfloor\tfrac12\,d(v)\bigr\rfloor\cdot1/4-31\,d(v)\,\alpha.\]
  If $d(v)=7$, this gives $\sigma_2(v)\ge1/4-217\,\alpha$; and if
  $d(v)\ge8$, we have the estimate
  $\sigma_2(v)\ge d(v)\cdot(7/8-31\,\alpha)-6$.

  We see that $\sigma_2(v)\ge\alpha>0$ for all $v\in Y$ if we can choose
  $\alpha>0$ such that $1/4-217\,\alpha\ge\alpha$ and
  $d\cdot(7/8-31\,\alpha)-6\ge\alpha$ for all $d\ge8$. It is easy to check
  that $\alpha=1/872$ will do the job. That value will also guarantee that
  $\sigma_2(v)>-3-93\,\alpha$ for all $v\in X$, completing the proof of the
  claim.\eop

  \bigskip\noindent
  The claim means that
  $-12=\sum\limits_{v\in V}\sigma_2(v)\ge(-3-93\,\alpha)\,|X|+\alpha\,|Y|$.
  This gives $|Y|\le(93+3/\alpha)\,|X|=2709\,|X|$. So the surviving rate of
  a graph on $n=|X|+|Y|$ vertices with this strategy is at least
  \[\frac{n-6}{n}\cdot\frac{|X|}{|X|+|Y|}\:>\:
  \frac{n-6}{n}\cdot\frac{|X|}{2710\,|X|}\:=\:\frac{n-6}{2710\,n}\hskip1pt.\]
  So if $n\ge10846$, the surviving rate is at least 1/2712. On the other
  hand, if $2 \le n<10846$, then we still can save at least $\min(4,n-1)$
  vertices, hence the surviving rate in that case is still at least
  $4/10846>1/2712$.
\end{proof}

\noindent
Theorem~\ref{th:pla} gives the upper bound on $\mathit{ff}(\cP)$ in
Theorem~\ref{th:total}\,(1) and the lower bound follows from the graph
$K_{2,n}$ as considered earlier.

\section{Triangle-free planar graphs}\label{sec:bip}

\begin{theorem}\label{th:bip}\mbox{}\\*
  Every triangle-free planar graph $G$ with at least two vertices satisfies
  $\rho_2(G)\ge1/723636$.
\end{theorem}

\begin{proof}
  For a star $K_{1,n-1}$, $n\ge2$, we have $\rho_2(K_{1,n-1})\ge1/2$, so we
  can assume~$G$ is not a star.

  Next we can assume that $G$ is edge-maximal with the property of being
  triangle-free and planar, since adding edges to the graph can only make
  things more difficult for the firefighters. As a consequence it is not
  difficult to see that $G$ is connected and, using the assumption that~$G$
  is not a star, in fact~$G$ has no cut-vertex. This means the minimum
  degree is at least~2.

  We assume some fixed embedding of~$G$ in the plane. The embedding gives a
  circular order on the neighbours of a vertex. We use this order to talk
  about \emph{consecutive} neighbours. Since $G$ is 2-connected, the
  \emph{degree of a face} (the number of edges in a boundary walk of the
  face) is precisely the number of vertices incident with the face. We use
  $d(f)$ to denote the degree of a face.

  For $2\le d\le4$, let $X_d$ be the set of vertices $v$ of degree~$d$ such
  that if the fire starts at~$v$, the two firefighters have a strategy that
  saves at least $|V(G)|-18$ vertices; the other vertices of degree~$d$
  form the set~$Y_d$. For $d\ge 5$, $Y_d$ is the set of all vertices of
  degree~$d$. We set $x_d=|X_d|$ and $y_d=|Y_d|$,
  $X=\bigcup_{2\le d\le4}X_d$, and $Y=\bigcup_{d\ge2}Y_d$. Observe that
  every vertex~$v$ of degree~2 is in~$X_2$, since placing the firefighters
  on $v$'s neighbours saves all the vertices except $v$. Hence, $Y_2$ is
  empty.

  Two vertices $u$ and $v$ are \emph{4-opposite} if there is a face of
  degree 4 with boundary vertices $u,a,v,b$ in that order, where at least
  one of $a,b$ has degree~4. Two vertices $u$ and $v$ are \emph{4-adjacent}
  if they are adjacent and the two faces incident with the edge $uv$ have
  degree 4.

  We now give some remarks about the set $X_3$.

  \begin{observation}\label{obs:3}
    Every vertex $v$ of degree 3 satisfying at least one of the following
    properties is in $X_3$:

    \medskip\noindent
    3.1\quad$v$ is adjacent to a vertex of degree at most 3;

    \smallskip\noindent
    3.2\quad$v$ is adjacent to a vertex of degree 4 having another
    neighbour of degree at most 3;

    \smallskip\noindent
    3.3\quad$v$ is 4-opposite to a vertex of degree at most 4;

    \smallskip\noindent
    3.4\quad$v$ is 4-opposite to a vertex of degree 5 having a neighbour of
    degree at most 3;

    \smallskip\noindent
    3.5\quad$v$ is adjacent to a vertex $w$ of degree 5 having three
    consecutive neighbours of degree at most~3 (including $v$), and such
    that the middle vertex from these three neighbours of~$w$ is 4-adjacent
    to~$w$;

    \smallskip\noindent
    3.6\quad$v$ is 4-adjacent to a vertex of degree 6 that is 4-adjacent to
    6 vertices of degree at most~3.
  \end{observation}

  \noindent
  To see this, consider Figures~\ref{fig:c3.123}, \ref{fig:c3.4},
  and~\ref{fig:c3.56}. In each figure, the strategy of the firefighters is
  described in the same way we did for Figure~\ref{fig:gridhex} (for
  instance, the firestart $v$ is the squared vertex labelled 0). The
  degrees of the relevant vertices are indicated next to those vertices.
  Note that the exact order of the neighbours around the vertex of degree~4
  and~5 in configurations~3.2 and~3.4, respectively, is not relevant and
  does not influence the strategy.\eop

  \begin{figure}[ht]
    \centering
    \raisebox{2.5mm}{\includegraphics[scale=0.6]{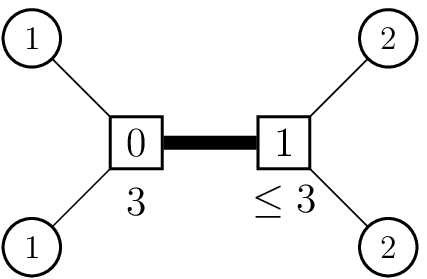}}
    \hspace{1cm}
    \includegraphics[scale=0.6]{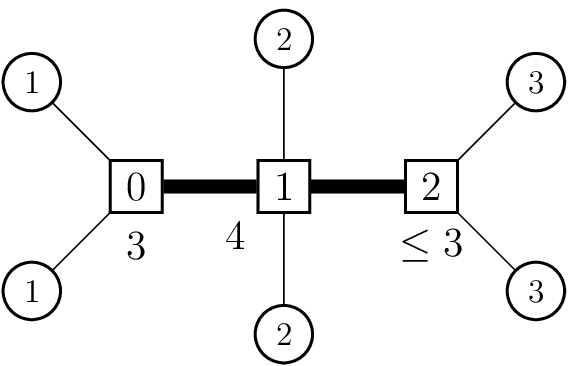}
    \caption{Configurations 3.1 (left) and 3.2 (right).}
    \label{fig:c3.123}
  \end{figure}

  \begin{figure}[ht]
    \centering
    \includegraphics[scale=0.6]{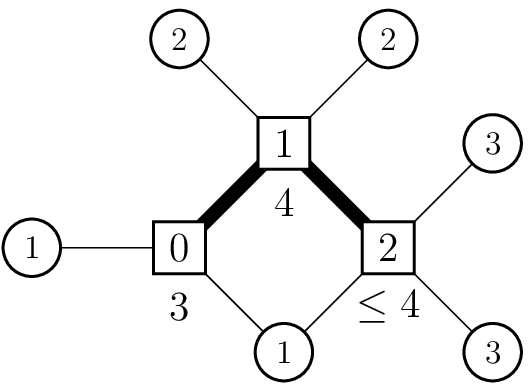}
    \hspace{1cm}
    \includegraphics[scale=0.6]{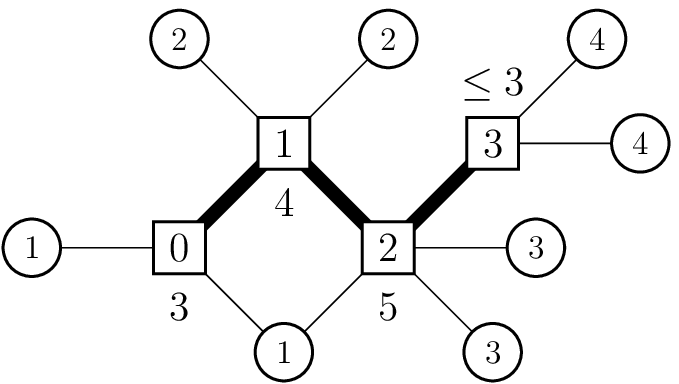}    
    \caption{Configurations 3.3 (left) and 3.4 (right).}
    \label{fig:c3.4}
  \end{figure}

  \begin{figure}[ht]
    \centering
    \raisebox{3mm}{\includegraphics[scale=0.6]{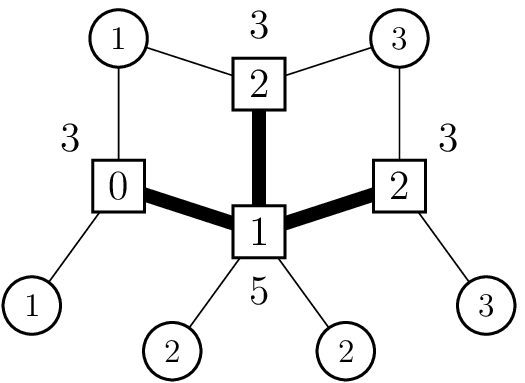}}
    \hspace{1cm}
    \raisebox{3mm}{\includegraphics[scale=0.6]{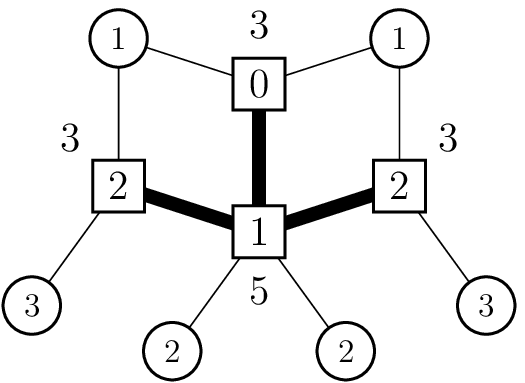}}
    \hspace{1cm}
    \includegraphics[scale=0.5]{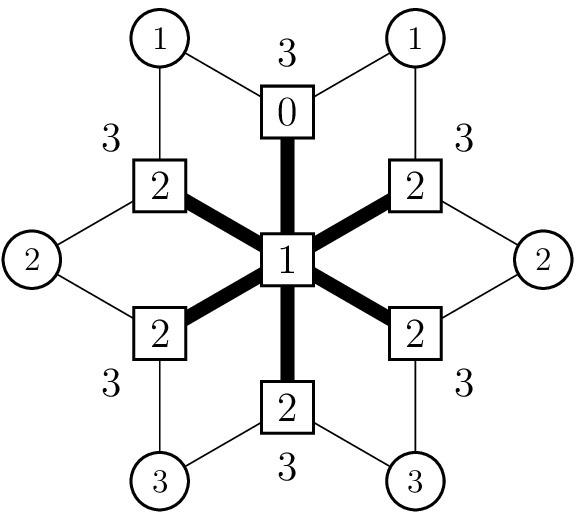}
    \caption{Configurations 3.5 (left and centre) and 3.6 (right).}
    \label{fig:c3.56}
  \end{figure}

  \bigskip\noindent
  Using these observations, we now derive some useful properties of the
  vertices in $Y_3$. We need a few more definitions. An \emph{element} is a
  vertex or a face. An element is \emph{contiguous} with a vertex $v$, if
  it is either a face that is incident with $v$, or a vertex that is
  4-opposite or 4-adjacent to $v$. Two vertices $u$ and $v$ are
  \emph{5-adjacent} if they are adjacent and exactly one of the two faces
  incident with the edge $uv$ has degree 4.

  \begin{claim}\label{cl:35}
    For any $v\in Y_3$, at least two elements of degree at least 5 are
    contiguous with~$v$. Moreover, if there are only two such elements,
    then $v$ is 5-adjacent to two more vertices of degree at least 5
    and~$v$ is incident with at least one face of degree at least~5 and at
    least one face of degree~4.
  \end{claim}

  \noindent
  By 3.1 of Observation~\ref{obs:3}, all the neighbours of a vertex
  $v\in Y_3$ have degree at least 4. If~$v$ is adjacent to a vertex $u$ of
  degree 4, then, by 3.3 of Observation~\ref{obs:3}, for each face~$f$
  incident with $uv$, either $d(f)\ge 5$ or there exists a vertex $w$
  incident with $f$ such that $d(w)\ge 5$ and $v$ and $w$ are 4-opposite.
  So if $v$ is incident to zero or three faces of degree 4, then $v$ is
  contiguous with at least three elements of degree at least~5, and the
  claim holds.

  If $v$ is incident to only one face of degree~4, then it is contiguous
  with two faces of degree at least~5. Moreover, by the above, either $v$
  is 4-opposite (and so, contiguous) to a vertex of degree at least~5, or
  it is 5-adjacent to two vertices of degree at least~5.

  Finally, assume that $v$ is incident to precisely two faces of degree 4.
  Then it is contiguous with a face $f$ of degree at least 5. Let $u$ be
  the neighbour of $v$ such that $uv$ is not incident with $f$. By the
  above, if $d(u)=4$, then $v$ is 4-opposite to two vertices of degree at
  least~5. So in this case it is contiguous with at least three elements of
  degree at least~5. If $d(u)\ge 5$, then, since $u$ and $v$ are
  4-adjacent, $v$ is contiguous with two elements of degree at least~5. By
  the remark above, if no other element of degree at least~5 is contiguous
  with $v$, then the two neighbours of $v$ distinct from $u$ must have
  degree at least~5, which completes the proof of the claim.\eop

  \bigskip\noindent
  We denote by $Y_{3,2}$ the vertices in~$Y_3$ that are only contiguous
  with two elements of degree at least 5.

  \begin{claim}\label{cl:53}
    For any $v\in Y_5$, at most three vertices of $Y_3$ are contiguous with
    $v$. Moreover, if three vertices of $Y_3$ are contiguous with $v$, they
    consist of three non-consecutive neighbours of $v$, and all the faces
    incident with $v$ have degree 4.
  \end{claim}

  \noindent
  We first observe that by 3.4 of Observation~\ref{obs:3}, a vertex
  $v\in Y_5$ cannot be both adjacent to a vertex $u\in Y_3$ and 4-opposite
  to a vertex $w\in Y_3$ (since otherwise $w$ would be in $X_3$ by
  definition). Assume that $v$ is 4-opposite to two consecutive
  vertices~$x$ and $y$ of $Y_3$ (\emph{i.e.}, the faces that $v$ shares
  with $x$ and $y$ are consecutive with respect to $v$). By 3.1 and 3.2 of
  Observation~\ref{obs:3}, the common neighbour $z$ of $v,x,y$ cannot have
  degree less than~5, so the two neighbours of $v$ distinct from $z$, but
  adjacent to $x$ or $y$, have degree~4. The situation, together with a
  strategy for the firefighters in the case a fire starts at~$x$, is
  depicted in Figure~\ref{fig:35cons}.
  \begin{figure}[ht]
    \centering
    \includegraphics[scale=0.7]{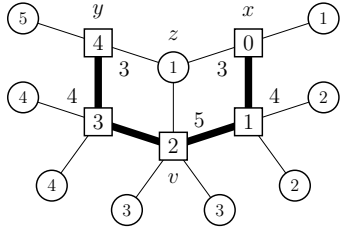}
    \caption{A vertex $v\in Y_5$ with two consecutive 4-opposite vertices
      $x,y\in Y_3$.}
    \label{fig:35cons}
  \end{figure}
  This contradicts $x\in Y_3$. It follows that $v$ cannot be 4-opposite to
  three vertices of $Y_3$. Assume now that~$v$ is adjacent to some vertices
  of $Y_3$. By 3.5 of Observation~\ref{obs:3}, there cannot be three
  vertices of $Y_3$ that are 4-adjacent to $v$ and consecutive around $v$.
  Moreover, if three non-consecutive neighbours of $v$ are 4-adjacent to
  $v$, then all the faces incident with $v$ have degree 4, which concludes
  the proof of Claim~\ref{cl:53}.\eop

  \bigskip\noindent
  We denote by $Y_{5,3}$ the vertices in~$Y_5$ that are 4-adjacent to three
  vertices of $Y_3$. By the previous claim, we see that all the faces
  incident with a vertex of~$Y_{5,3}$ have degree 4, but a vertex of $Y_3$
  cannot be 4-opposite to a vertex of~$Y_{5,3}$.

  \begin{claim}\label{cl:353}
    Every vertex in $Y_3$ is 4-adjacent to at most one vertex of $Y_{5,3}$.
  \end{claim}

  \noindent
  Assume the claim is false, and some vertex $v in Y_3$ is 4-adjacent to
  two vertices of $Y_{5,3}$. By Claim~\ref{cl:53}, only four cases need to
  be considered. In each of the cases, there is a strategy for the
  firefighters for a fire that starts at one of the vertices from~$Y_3$ and
  saves at least $|V(G)|-7$ vertices; see Figure~\ref{fig:y53}. But this
  contradicts the definition of a vertex in~$Y_3$.\eop

  \begin{figure}[ht]
    \centering
    \includegraphics[scale=0.5]{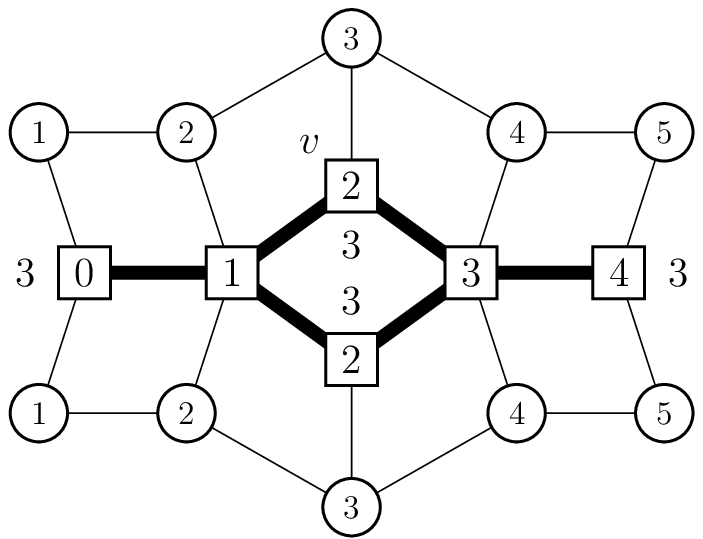}\hspace{0.4cm}
    \includegraphics[scale=0.5]{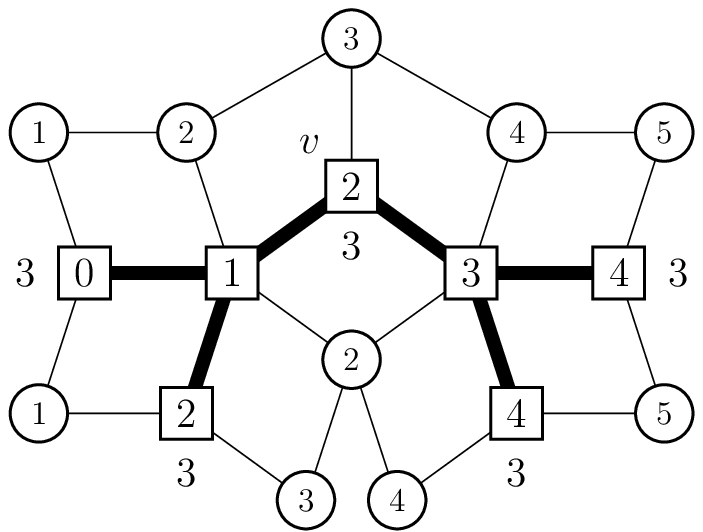}\hspace{0.4cm}
    \includegraphics[scale=0.5]{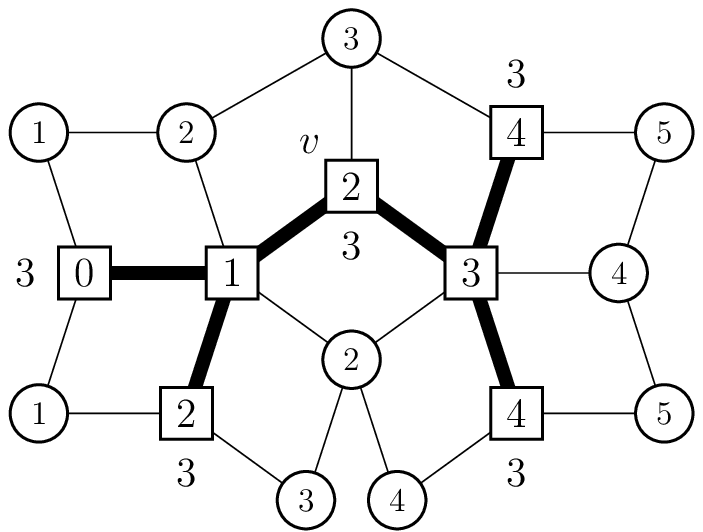}\hspace{0.4cm}
    \includegraphics[scale=0.5]{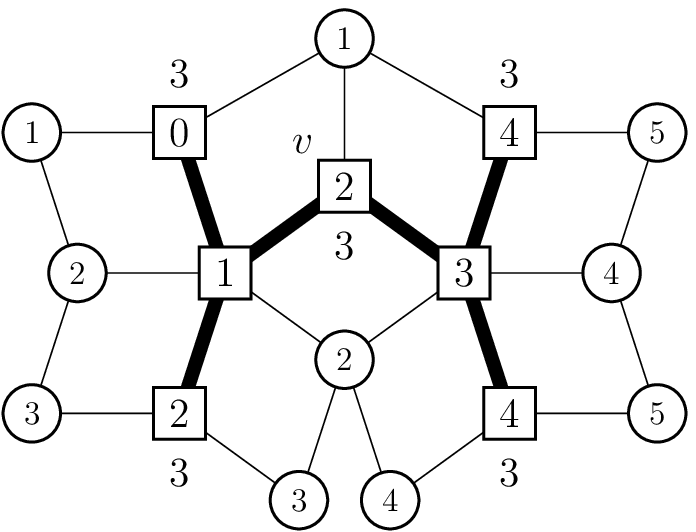}
    \caption{A vertex $v\in Y_3$ with two neighbours in $Y_{5,3}$.}
    \label{fig:y53}
  \end{figure}

  \begin{claim}\label{cl:6}
    If $v\in Y_6$, then at most six vertices of $Y_3$ are contiguous with
    or 5-adjacent to~$v$. Moreover, if there are six such vertices, $v$ is
    5-adjacent to at least two of them.
  \end{claim}

  \noindent
  By 3.1 of Observation~\ref{obs:3}, if $v$ is 4-opposite to some vertex
  $u\in Y_3$, then the (at least two) common neighbours of $v$ and $u$ have
  degree at least 4, so they are not in $Y_3$. Assume that six or more
  vertices of~$Y_3$ are contiguous with or 5-adjacent to $v$. By the remark
  above, this can only happen if either~$v$ is 4-opposite to six vertices
  of $Y_3$, which contradicts 3.2 of Observation~\ref{obs:3}, or $v$ is
  4-adjacent or 5-adjacent to six vertices of $Y_3$. In the latter case, by
  3.6 of Observation~\ref{obs:3}, $v$ is not allowed to be 4-adjacent to
  six vertices of~$Y_3$, so~$v$ must be incident with a face of degree at
  least~5. Since not all faces can have degree at least~5 (since then~$v$
  cannot be 4-adjacent or 5-adjacent to any vertex), $v$ is 5-adjacent to
  at least two vertices of $Y_3$.\eop

  \bigskip\noindent
  The argument above involving 3.1 has the following consequence in
  general.

  \begin{claim}\label{cl:7}
    If $v\in Y_d$ with $d\ge 7$, at most $d$ vertices of $Y_3$ are
    contiguous with or 5-adjacent to $v$.
  \end{claim}

  \noindent
  We finish this part with one observation regarding vertices in~$Y_4$.

  \begin{observation}\label{obs:4}
    For every vertex $v\in Y_4$, there is a path of length at most 7
    connecting~$v$ and a vertex $u$ such that either $d(u)\ne4$, or a face
    incident with $u$ and its neighbour on the path has degree at least~5.
    Moreover, all the internal vertices on the path and the faces incident
    with $v$ or with two internal vertices have degree precisely 4.
  \end{observation}

  \noindent
  Assume this is not the case. Then the subgraph of $G$ induced by the
  vertices at distance at most 7 from $v$ is the induced subgraph of a
  rectangular grid. In this case Fogarty~\cite{Fog03} gave a strategy
  saving all the vertices except at most 18 when the fire starts at $v$
  (see Figure~\ref{fig:grid}). This contradicts the fact that
  $v\in Y_4$.\eop

  \begin{figure}[ht]
    \centering
    \includegraphics[scale=0.8]{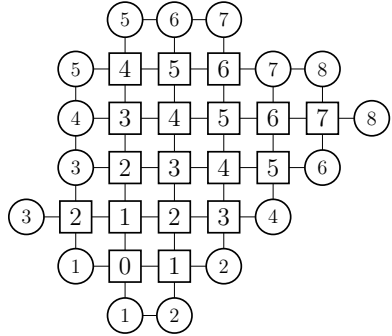}
    \caption{A strategy saving at least $n-18$ vertices if the
      neighbourhood of the firestart is a rectangular
      grid.\label{fig:grid}}
  \end{figure}

  \bigskip\noindent
  We continue as in the proof of Theorem~\ref{th:pla}. We assign a charge
  $\nu_1(v)=d(v)-4$ to each vertex $v\in V$, and a charge $\nu_1(f)=d(f)-4$
  to each face $f\in F$. Euler's formula gives
  $\sum\limits_{v\in V}\nu_1(v)+\sum\limits_{f\in F}\nu_1(f)=-8$. We
  redistribute this initial charge according to the following rules. Here
  the values of~$\alpha$ and $\beta$ will be determined later.

  \medskip\noindent
  (S1)\quad A vertex in~$Y_{5,3}$ gives a charge of $1/3-\beta$ to each of
  its three neighbours from~$Y_{3}$.

  \smallskip\noindent
  (S2)\quad A vertex of degree at least 5 not in $Y_{5,3}$ gives a charge
  of $2/5-\beta$ to each vertex in~$Y_3$ it is contiguous with.

  \smallskip\noindent
  (S3)\quad A vertex of degree at least~5 gives a charge of $1/10-\beta$ to
  each vertex in~$Y_3$ it is 5-adjacent to.

  \smallskip\noindent
  (S4)\quad Each face of degree at least~5 gives a charge of $1/2-\beta$ to
  each vertex in~$Y_3$ it is incident with.

  \smallskip\noindent
  (S5)\quad For each vertex $v\in Y_4$ we choose a vertex~$u$ with
  $\mathop{dist}(v,u)\le7$ according to Observation~\ref{obs:4}, such that
  $\mathop{dist}(v,u)$ is minimal. If $d(u)\ne4$, this vertex~$u$ gives a
  charge~$\alpha$ to~$v$. Otherwise, $u$ is incident to a face of degree at
  least 5, which is at distance at most 6 from $v$; in that case this face
  gives a charge~$\alpha$ to~$v$.

  \medskip
  The charge obtained after applying rules (S1)\,--\,(S4) is denoted by
  $\nu_2(x)$, $x\in V\cup F$. Note that we have
  $\sum\limits_{v\in V}\nu_2(v)+\sum\limits_{f\in F}\nu_2(f)=
  \sum\limits_{v\in V}\nu_1(v)+\sum\limits_{f\in F}\nu_1(f)=-8$.
  
  \begin{claim}\label{cl:chb}
    Let $\alpha=1/360720$ and $\beta=2186\,\alpha$. Then for every $v\in X$
    we have $\nu_2(v)\ge-2-\beta$; for every $v\in Y$ we have
    $\nu_2(v)\ge\alpha$; and for every face $f$, $\nu_2(f)\ge 0$.
  \end{claim}

  \noindent
  To prove the claim, we first estimate how often a vertex~$v$ with
  $d(v)\ne4$ can give a charge~$\alpha$ according to~(S5). As a very crude
  upper bound, this is at most the number of paths of length at most~7,
  starting in~$v$, and whose internal vertices all have degree exactly~4.
  This number is clearly at most
  $d(v)\cdot(1+3+3^2+\cdots+3^6)=1093\,d(v)$. Hence, each vertex of degree
  $d$ gives at most $1093\,d\,\alpha=\frac12\,d\,\beta$ according to~(S5).

  It follows that a vertex~$v$ of degree~2 satisfies
  $\nu_2(v)\ge\nu_1(v)-\beta=-2-\beta$. Similarly, for a vertex $v\in X_3$
  we have $\nu_2(v)\ge-1-\frac32\,\beta\ge-2-\beta$, since $\beta\le2$. And
  for $v\in X_4$ we have $\nu_2(v)=\nu_1(v)=0$.

  Let $v\in Y_{3}$. Suppose first that $v$ is contiguous with three
  elements of degree at least~$5$. By Claim~\ref{cl:353}, at most one of
  them is in $Y_{5,3}$, so $v$ receives at least
  $1/3+2\times2/5-3\,\beta=17/15-3\,\beta$ by rules~(S1), and~(S2).
  Otherwise, by Claim~\ref{cl:35}, $v$ is only contiguous with two elements
  of degree at least 5, i.e., $v\in Y_{3,2}$. In this case we know that $v$
  is incident to a face~$f$ of degree at least~$5$ and 5-adjacent to two
  vertices of degree at least 5. Hence $v$ receives at least
  $1/2+1/3-2\,\beta$ from $f$ and the second element of degree at least
  five it is contiguous with (by rules~(S4), and~(S1) or~(S2)), and at
  least $2\times (1/10-\beta)$ by rule~(S3). So, in both cases $v$ receives
  at least $31/30-4\,\beta$. Since it gives at most $\frac32\,\beta$
  according to~(S5), we obtain $\nu_2(v)\ge 1/30-\frac{11}2\,\beta$. Using
  that $\alpha=1/360720$ and $\beta=2186\,\alpha$, this implies
  $\nu_2(v)\ge\alpha$.

  From rule~(S5) it follows that if $v\in Y_4$, then
  $\nu_2(v)=\nu_1(v)+\alpha=\alpha$.

  Next consider a vertex $v\in Y_{5,3}$. This gives three times a charge
  $1/3-\beta$ according to~(S1), no charge according to (S3) by
  Claim~\ref{cl:53}, and at most $5\times\frac12\,\beta$ according to~(S5).
  Since $\nu_1(v)=1$, we find
  $\nu_2(v)\ge1-3\times(1/3-\beta)-\frac52\,\beta=\frac12\,\beta$.

  Consider now a vertex $v\in Y_5\setminus Y_{5,3}$. By Claim~\ref{cl:53},
  $v$ is contiguous with at most two vertices from~$Y_3$. Recall that by
  3.4 of Observation~\ref{obs:3}, $v$ cannot be 4-opposite to a vertex
  of~$Y_3$ and adjacent to a vertex of $Y_3$. If $v$ is not contiguous with
  any vertex of degree at least~5, it gives at most $5\times(1/10-\beta)$
  according to~(S3). If $v$ is contiguous with one vertex of~$Y_3$, then it
  gives $2/5-\beta$ according to (S2) and at most $4\times(1/10-\beta)$
  according to~(S3). Suppose now that $v$ is contiguous with exactly two
  vertices from $Y_3$, say $u$ and~$w$. If one of them in 4-opposite
  to~$v$, then $v$ gives $2\times(2/5-\beta)$ only. Otherwise, $v$ is
  4-adjacent with $u$ and $w$. If $v$ is 5-adjacent to its three other
  neighbours, then we end up with configuration~3.5 of
  Observation~\ref{obs:3}, a contradiction. Hence $v$ can only be
  5-adjacent to two more vertices, and it gives $2\times(2/5-\beta)$
  according to (S2) and at most $2\times(1/10-\beta)$ according to (S3). In
  all cases, $v$ gives at most $1-4\,\beta$. Since $v$ gives at most
  $5\times\frac12\,\beta$ according to (S5), we have
  $\nu_2(v)\ge1-(1-4\,\beta)-\frac52\,\beta=\frac32\,\beta$.
  
  For a vertex $v\in Y_6$, we have by Claim~\ref{cl:6} that it gives at
  most five times a charge according to~(S2). Moreover, if it gives
  precisely $5\times(2/5-\beta)$, then it is not 5-adjacent to any vertex
  of $Y_3$. Otherwise, it gives at most
  $4\times(2/5-\beta)+2\times(1/10-\beta)$. Hence, $v$ gives at most
  $\max\{2-5\,\beta,9/5-6\,\beta\}=2-5\,\beta$. Finally, $v$ also gives at
  most $6\times\frac12\,\beta$ according to~(S5). As $\nu_1(v)=2$, we
  obtain $\nu_2(v)\ge2-(2-5\,\beta)-3\,\beta=2\,\beta$.

  For a vertex~$v$ with $d(v)\ge7$, we can estimate, using Claim~\ref{cl:7}
  and the observations above:
  \[\nu_2(v)\:\ge\:(d(v)-4)-d(v)\cdot(2/5-\beta)-
  d(v)\cdot\tfrac12\,\beta\:=\:
  d(v)\cdot(3/5+\tfrac12\,\beta)\:\ge\:\tfrac72\,\beta.\]

  We now estimate how often a face $f$ can give a charge~$\alpha$ according
  to~(S5). This is at most the number of paths of length at most~6,
  starting at a vertex $u$ of degree 4 incident with~$f$, and whose
  internal vertices (and faces incident with them) all have degree
  exactly~4. This number is at most
  $2\,d(f)\cdot(1+3+3^2+\cdots3^5)=728\,d(f)$, since such a vertex $u$ has
  at most two neighbours such that the faces incident with those neighbours
  all have degree 4. Using that $\beta=2186\,\alpha$, this implies that
  each face of degree $d$ gives at most
  $728\,d\,\alpha\le{\tfrac25}\,d\,\beta$ according to (S5).

  By 3.1 of Observation~\ref{obs:3}, a face $f$ has at most
  $\bigl\lfloor\tfrac12\,d(f)\bigr\rfloor$ vertices of $Y_3$ on its
  boundary. Hence, for a face~$f$ of degree~5, we have
  $\nu_2(f)\ge1-2\times(1/2-\beta)-5\times\frac25\,\beta=0$. For a face $f$
  of degree at least 6, we obtain
  \[\nu_2(f)\:\ge\:(d(f)-4)-\tfrac12\,d(f)\cdot(1/2-\beta)-
  d(f)\cdot\tfrac25\,\beta\:=\:
  d(f)\cdot(3/4+\tfrac1{10}\,\beta)-4\:\ge\:\tfrac35\,\beta.\]

  Finally, for a face $f$ of degree 4, we have $\nu_2(f)=\nu_1(f)=0$.

  Putting the inequalities together, we see that $\nu_2(v)\ge-2-\beta$ for
  all $v\in X$; $\nu_2(v)\ge\min\{\alpha,\frac12\,\beta\}=\alpha$ for all
  $v\in Y$; and $\nu_2(f)\ge 0$ for any face $f$, completing the proof of
  the claim.\eop
  
  \bigskip\noindent
  The claim means that
  $-8=\sum\limits_{v\in V}\nu_2(v)+ \sum\limits_{f\in F}\nu_2(f)\ge
  (-2-\beta)\,|X|+\alpha\,|Y|$. This gives
  $|Y|<(2186+2/\alpha)\,|X|=723626\,|X| $. So the surviving rate of a graph
  on $n=|X|+|Y|$ vertices with this strategy is at least
  \[\frac{n-18}{n}\cdot\frac{|X|}{|X|+|Y|}\:>\:
  \frac{n-18}{n}\cdot\frac{|X|}{723627\,|X|}\:=\:
  \frac{n-18}{723627\,n}\hskip1pt.\]
  So if $n\ge1447272$, the surviving rate is at least 1/723636. On the
  other hand, if $2\le n<1447272$, then we still can save at least
  $\min\{2,n-1\}$ vertices, hence the surviving rate in that case is still
  at least $2/1447272=1/723636$.
\end{proof}

\noindent
Again, we have made no attempts to optimise the constants, in order to
concentrate on making the exposition as clear as possible. For instance,
the estimates for the number of times a vertex or a face gives a
charge~$\alpha$ according to rules~(S5) can be improved significantly with
a more careful analysis.

Theorem~\ref{th:bip} together with the graph $K_{2,n}$ imply
Theorem~\ref{th:total}\,(2).

\section{Conclusion}

Regarding the firefighter number of planar graphs, we conjecture the
following.

\begin{conjecture}\mbox{}\\*
  For the class $\cP$ of planar graphs, we have $\mathit{ff}(\cP)=2$.
\end{conjecture}

\noindent
We believe that the proof of Theorem~\ref{th:pla} can be modified to prove
that for some $\epsilon>0$, every planar graph $G$ satisfies
$\rho_3(G)\ge\epsilon$, which would give $\mathit{ff}(\cP_5)\le3$. The only
difference with our setting is that only three firefighters are available
at the first round (instead of four\,\footnote{\,It was pointed out by one
  of the referees that adding one firefighter at the first round can have
  significant implications in terms of complexity. In the usual setting
  (only one firefighter at each round), deciding whether a given number of
  vertices of a rooted subcubic tree can be saved if the fire starts at the
  root is NP-complete~\cite{FKMR07}. If a second firefighter is available
  at the first round only, the problem becomes trivially polynomial.}). The
main consequence is that not only $Y_4$ is not empty, but also the
structure of~$Y_5$ is more complicated. However, the argument concerning
$Y_6$ will run smoothly using a strategy of Fogarty~\cite{Fog03} in
hexagonal grids.

Nevertheless, we feel the proof technique is too local to lead to a proof
of $\mathit{ff}(\cP)=2$. For instance, the fact that only a constant number
of steps are considered does not allow to design a good strategy (saving a
linear number of vertices in average) with two firefighters in a very large
hexagonal grid.

\medskip
Wang \emph{et al.}~\cite{WFW10} proved that $\mathit{ff}(\cP_9)=1$. The
ideas of the proof of Theorem~\ref{th:bip} can be adapted to prove that
$\mathit{ff}(\cP_8)=1$, since it is not hard to show that a worst case
scenario in this case is if locally, every face close to the firestart has
size 8, and around the boundary of these faces the vertices alternatingly
have degree~2 and~4. (In other words, locally the graph looks like a
subdivided rectangular grid.) In this case, a strategy similar to the
strategy of Fogarty described in Section~\ref{sec:bip} will save at least
$n-63$ vertices. We omit details here, since we believe that such a result
would still be far from optimal. Indeed, we conjecture the following.

\begin{conjecture}\mbox{}\\*
  For the class $\cP_5$ of planar graphs of girth at least 5, we have
  $\mathit{ff}(\cP_5)=1$.
\end{conjecture}

\noindent
We finish with a remark on the connection between separators and
firefighters. For some constant $\epsilon>0$, an
\emph{$\epsilon$-separator} $S$ in a connected graph $G$ is a set of
vertices whose removal yields at least two components of size at least
$\epsilon |V(G)|$. If for $\epsilon |V(G)|$ vertices $v$ of~$G$, there is
an $\epsilon$-separator $S_v$ whose cardinality is at most the distance
between $v$ and~$S_v$, then $\rho_1(G)>\epsilon^2$: by the time a fire
starting at $v$ reaches $S_v$, the single firefighter can protect all the
vertices of $S_v$, so a (linear-sized) component of $G \setminus S_v$ not
containing $v$ will be saved.

\subsection*{Acknowledgement}

The authors thank the anonymous referees for comments and suggestions. We
also would like to thank Jiangxu Kong from Zhejiang Normal University,
China for pointing out an error in an earlier version of the proof of
Theorem~\ref{th:bip}.

\end{document}